\documentclass[letterpaper, 10 pt, conference]{ieeeconf}  

\IEEEoverridecommandlockouts                              
\usepackage{epsfig} 
\usepackage{amsmath,amssymb,amsfonts}
\usepackage{color}
\usepackage{balance}
\usepackage{booktabs,tabularx}
\usepackage{multirow}
\usepackage{mathtools}
\usepackage{pdfrender}
\usepackage{trfsigns}

\usepackage{centernot}
\usepackage{color}
\usepackage[acronym]{glossaries}
\usepackage[ruled]{algorithm2e}
\usepackage{graphicx,subcaption}
\usepackage{tabu}
\usepackage[official]{eurosym}

\newcommand{\R}{\mathbb{R}}

\newcommand{\B}{\mathbb{B}}

\newcommand{\N}{\mathbb{N}}

\newcommand{\mc}[1]{\mathcal{#1}}

\newcommand{\bs}{\boldsymbol}

\newcommand{\bld}[1]{\boldsymbol{#1}}
\newcommand{\col}{\mathrm{col}}

\newtheorem{theorem}{Theorem}

\newtheorem{definition}{Definition}

\newtheorem{corollary}{Corollary}

\makeglossaries
\newacronym{MIP}{MIP}{Mixed-Integer Programming}
\newacronym{SoC}{SoC}{State of Charge}
\newacronym{PEV}{PEV}{Plug-in Electric Vehicle}
\newacronym{EV}{EV}{Electric Vehicle}
\newacronym{MLD}{MLD}{Mixed-Logical-Dynamical}
\newacronym{MPC}{MPC}{Model Predictive Control}
\newacronym{GS}{GS}{Gauss-Southwell}
\newacronym{GNEP}{GNEP}{Generalized Nash Equilibrium Problem}
\newacronym{MI-GPG}{MI-GPG}{Mixed-Integer Generalized Potential Game}
\newacronym{BSI}{BSI}{Bilateral Symmetric Interaction}
\newacronym{MINE}{$\varepsilon$-MINE}{$\varepsilon$-Mixed-Integer Nash Equilibrium}

\title{\LARGE \bf
Charging plug-in electric vehicles as a mixed-integer aggregative game
}

\author{Carlo Cenedese$^{\textrm{(a)}}$, Filippo Fabiani$^{\textrm{(b)}}$, Michele Cucuzzella$^{\textrm{(a)}}$, Jacquelien M. A. Scherpen$^{\textrm{(a)}}$,\\ Ming Cao$^{\textrm{(a)}}$ and Sergio Grammatico$^{\textrm{(b)}}$ 
\thanks{$^{\textrm{(a)}}$ Jan C. Wilems Center for Systems and Control, ENTEG, Faculty of Science and Engineering, University of Groningen, The Netherlands	({\texttt{\{c.cenedese, m.cucuzzella, j.m.a.scherpen, m.cao\}@rug.nl}}). 
$^{\textrm{(b)}}$ Delft Center for Systems and Control, TU Delft, The Netherlands
	({\texttt{\{f.fabiani, s.grammatico\}@tudelft.nl}})
	The work of Cenedese and Cao was supported by the Netherlands Organization for Scientific Research (NWO-vidi-14134), the one of Cucuzzella and Scherpen by the EU Project \lq MatchIT' (82203), and the one of Fabiani and Grammatico by NWO under project OMEGA (613.001.702) and P2P-TALES (647.003.003) and by the ERC under research project COSMOS (802348).}
}

\begin{document}

\maketitle
\thispagestyle{empty}
\pagestyle{empty}

\begin{abstract}

We consider the charge scheduling coordination of a fleet of plug-in electric vehicles, developing a hybrid decision-making framework for efficient and profitable usage of the distribution grid. Each charging dynamics, affected by the aggregate behavior of the whole fleet, is modelled as an inter-dependent, mixed-logical-dynamical system. The coordination problem is formalized as a generalized mixed-integer aggregative potential game, and solved via semi-decentralized implementation of a sequential best-response algorithm that leads to an approximated equilibrium of the game.
\end{abstract}

\section{Introduction}
\glspl{PEV} represent a promising alternative to the conventional fuel-based transportation \cite{Review_2016}. However, they require a ``smart'' charge/discharge coordination to prevent undesired electricity demand peaks and congestion in the low-voltage distribution grid~\cite{doi:10.1002/etep.565,5356176}. Indeed, 
a fleet of \glspl{PEV} can serve as a \emph{mobile extension} of the grid, mitigating the intermittent behavior of renewable energy sources by storing (or providing) energy when the generation is higher (or lower) than the load demand. These capabilities are known as `` grid-to-vehicle" (g2v) and ``vehicle-to-grid" (v2g), respectively \cite{6021358}.
In this context, a key role is played by the ``fleet manager'' or ``aggregator'', which is partially responsible for providing charging services to the \glspl{PEV}, coordinating the charge/discharge schedule to fulfill the needs of the \glspl{PEV} owners and (desirably) providing ancillary services to the distribution system operator \cite{Review_2016, doi:10.1002/etep.565}.

Accordingly, an attracting framework that suitably condenses both individual interests and intrinsic limitations related with the shared facilities is represented by noncooperative game theory whose equilibrium solution, for instance in \cite{J_Hiskens_2013}, allows to coordinate the charging of a population of PEVs in a decentralized fashion with the aim of filling the night-time valley.
A similar non-cooperative equilibrium problem is addressed in \cite{C_Parise_2014} using mean field game theory, while
in \cite{Gan_2013} the authors generalize the notion of valley-filling and study a social welfare optimization problem associated with the charging scheduling, proposing an asynchronous optimization algorithm. The \gls{PEV} charging coordination game is further developed in \cite{ma2016efficient} by including the battery degradation cost within the cost function of each vehicle. 
More generally, a semi-decentralized coordination algorithm based on fixed-point operator theory is proposed in \cite{Grammatico_2016}, which achieves global exponential convergence to an aggregative equilibrium, even for large population size, while a distributed approach based on the progressive second price auction mechanism has been recently proposed in \cite{Zou_2017}.

However, most of the models adopted in the literature do not describe  the intrinsic discrete operations of each PEV, e.g., being plugged-in or plugged-out from one of the available charging points, in g2v or in v2g mode~\cite{J_You_2016,J_Xing_2016,J_Rivera_2017}. 
The selfish nature of each \gls{PEV}, which pursues an economic, possibly profitable, usage of the charging station, together with the presence of both continuous (amount of energy charged/discharged) and discrete decision variables (operating mode) over a certain prediction horizon, motivates us to model the charge scheduling coordination of the \glspl{PEV} as a collection of inter-dependent mixed-integer optimization problems, where the coupling constraints arise in aggregative form (\S\ref{sec:PEV_scheduling}). In fact, the common electricity price, as well as the operational limitations of the charging station, e.g., the number of charging points, directly depend on the overall behavior of the fleet.
We then formulate the coordination problem as a mixed-integer aggregative game (\S \ref{sec:logic_to_mix_integer}), proving that it corresponds to a \gls{MI-GPG} \cite{facchinei2011decomposition}. Furthermore, we assume the presence of an aggregator, i.e., the charging station, that provides the aggregate information to the fleet and allows to compute an (approximated) equilibrium of the game by means of a semi-decentralized implementation of a sequential best-response algorithm (\S \ref{sec:gen_mixed_int_pot_game}). Finally, we define the game setup and the solution algorithm via numerical simulations on an illustrative scenario (\S \ref{sec:num_sim}).


\emph{Notation: }
$\B$ represents the binary set. 
For vectors $v_1,\dots,v_N\in\mathbb{R}^n$ and $\mc I=\{1,\dots,N \}$, the collective strategy is denoted  as $\boldsymbol{v}:=\mathrm{col}((v_i)_{i\in\mc I})=[v_1 ^\top,\dots ,v_N^\top ]^\top$ and $ \bld v_{-i}:=\col(( v_j )_{j\in\mc I\setminus \{i\}})=[v_1^\top,\dots,v_{i-1}^\top,v_{i+1}^\top,\dots,v_{N}^\top]^\top$. With a slight abuse of notation, we also use $\bld v = (v_i,\bld{v}_{-i})$.

\section{PEVs scheduling and charge as a system of mixed-logical-dynamical systems}
\label{sec:PEV_scheduling}
We start by  modeling the charge scheduling of a \gls{PEV} over a certain prediction horizon by first introducing the dynamics of the \gls{SoC} and the related control variables, both discrete and continuous. Then, we define the cost function that each \gls{PEV} aims to minimize and, by means of logical implications, the interactions arising both among the set of \glspl{PEV} itself and with the charging station.

\subsection{Decision variables and the SoC dynamics}
Let  $\mc I\coloneqq\{1,\dots,N\}$ be the set indexing a population of  \glspl{PEV} that share the charging resources on a single charging station with a finite number of charging points, over a prediction horizon $\mc T\coloneqq\{1,\dots,T\}$. 
%
%
For each vehicle $i\in\mc I$, the \gls{SoC} of its battery at time $t\in\mc T$ is denoted by $x_i(t)\in [0,1]$, where $x_i(t)=1$ represents a fully charged battery, while $x_i(t)=0$ a completely discharged one. 
The amount of {energy} exchanged between each vehicle and the charging station, during the time period $t$, is denoted by $u_i(t)\in \mc U\coloneqq[\underline{u},\overline{u}]$, $\underline{u}<0$, $\overline{u}>0$.
 Specifically, $u_i(t)$ takes positive values if vehicle~$i$ absorbs power during the time period $t$, while negative values if it gives energy back. 
For all $ t \in \mc{T}$, we assume the dynamics of the \gls{SoC} described by 
\begin{equation} \label{eq:SoC_update_u}
x_i(t+1) = x_i(t) + b_i \delta_{i}(t) u_i(t) - (1-\delta_{i}(t)) \mu_i(t),
\end{equation}  
where $b_i=\frac{\eta_i}{C_{i}}$ is a positive parameter which depends on the efficiency $\eta_i > 0$ of the $i$-th battery and the corresponding capacity $C_{i} > 0$. In \eqref{eq:SoC_update_u}, $\delta_{i}(t)\in \B$ is a scheduling variable of each vehicle, i.e., $\delta_{i}(t)=1$ if the $i$-th \gls{PEV} is connected to the charging station at time $t$, $0$ otherwise. Within the dynamics~\eqref{eq:SoC_update_u}, we consider that the \gls{SoC} of the battery decreases when the vehicle is actively driving. In fact, if $\delta_i(t)=0$, the \gls{SoC} decreases according to $\mu_i(t)>0$, which represents the discharge due to the energy consumed while driving during the time period $t\in\mc T$. 
Therefore, $\mu_i(t)>0$ (driving vehicle) can only affects the \gls{SoC} when the \gls{PEV} is not connected to the grid, i.e., $\delta_i(t)=0$. If the vehicle is not driving during the time period $t$, then $\mu_i(t)=0$.

Thus, over the whole horizon $\mc T$, each \gls{PEV} $i \in \mc{I}$ has decision variables $u_i \coloneqq \col( (u_{i}(t))_{t\in\mc{T}})$, $\delta_i \coloneqq \col((\delta_i(t))_{t\in\mc{T}})$ and controls the \gls{SoC}, $x_i \coloneqq \col((x_i(t))_{t\in\mc{T}})$.


\subsection{Cost function}
We assume that when the vehicle is connected to the charging station, it can either acquire energy from the charging station, or give it back. 
If the $i$-th \gls{PEV} at time $t\in\mc T$ requires power from the charging station, i.e., $u_i>0$, then it faces a cost associated to the purchase and to the battery degradation. We model the first term as $c(d(t)+\sum_{j\in\mc I\setminus\{i\}} u_j^+(t))u_i(t)$, where we assume a fair electricity market with the cost per energy unit $c > 0$, the same for all the \glspl{PEV}, and $u_j^+(t)\coloneqq u_j(t)$ if $u_j(t)\geq 0$, $0$ otherwise.  The values  $(d(t))_{t\in\N}$ represent the non-\gls{PEV} loads of the network, which is assumed to be a-priori known, while $\sum_{i\in\mc I\setminus\{i\}} u^+_i(t) \eqqcolon a_i(\bs{u}(t))$ is the \emph{aggregate} demand of energy associated with the set of \glspl{PEV} at time $t$. 
Then, we assume that the degradation cost associated to the charging of the battery is proportional to the variation of the energy exchanged between the \gls{PEV} and the charging station, i.e., $\rho^+_i(u_i(t)-u_i(t-1))^2$, with $\rho^+_i > 0$ depending on each single battery. Therefore, for all $t \in \mc{T}$ and for all $i \in \mc{I}$, the cost function associated to the charging phase (g2v) is
$$
\begin{aligned}
J^{\textup{g2v}}_{i}(u_i(t), \bs{u}_{-i}(t)) \coloneqq c(d(t)+ &a_i(\bs{u}(t))) u_i(t)\\ & +\rho^+_i(u_i(t)-u_i(t-1))^2.
\end{aligned}
$$ 
Similarly, when the \gls{PEV} gives energy to the charging station, i.e., $u_i(t)<0$, the vehicle-to-grid (v2g) interaction implies a degradation cost due to the battery discharge, i.e., $\rho^-_i(u_i(t)-u_i(t-1))^2$, $\rho^-_i > 0$. However, this cost is generally compensated by a reward  $r_i(t) u_i(t)$, for some given values $(r_i(t))_{t\in\N}$, that the \gls{PEV} receives from the charging station for actively participating in the scheduling.
 Thus, the cost function related with the v2g mode reads as
$$
J^{\textup{v2g}}_{i}(u_i(t)) \coloneqq  r_i(t)u_i(t) + \rho^-_i (u_i(t)-u_i(t-1))^2 ,
$$
for all $t \in \mc{T}$ and $i \in \mc{I}$.

Now, we introduce two binary decision variables, i.e., $\delta^{\textup{c}}_{i}$, $\delta^{\textup{d}}_{i}\in\B$, which allow to model the mutual actions of charging/discharging and guarantee that the local cost function is equal to $0$ when $u_i = 0$. They shall satisfy the next logical implications, for all $t \in \mc{T}$ and for all $i \in \mc{I}$:
\begin{align}
	&[\delta^{\textup{c}}_{i}(t) = 1] \iff [u_i(t)\geq 0],\label{eq:logical_implication_deltar}\\
	&[\delta^{\textup{d}}_{i}(t)=1] \iff [u_i(t) \leq 0].\label{eq:logical_implication_sigma}
\end{align}

{Namely,  $\delta^{\textup{c}}_{i}(t)$ (or, equivalently, $\delta^{\textup{d}}_{i}(t)$) is equal to $1$ if  the \gls{PEV} is not discharging (not charging), $0$ otherwise.} Finally, for all $i \in \mc{I}$, the complete cost function reads as follows
\begin{equation} \label{eq:cost_function}
\begin{aligned}
J_i(u_i, \delta^{\textup{d}}_{i}, \delta^{\textup{c}}_{i}, \bs{u}_{-i}) \coloneqq  \textstyle\sum_{t\in\mc T} &J^{\textup{g2v}}_{i}(u_i(t), \bs{u}_{-i}(t)) \left(1 - \delta^{\textup{d}}_{i}(t)\right)\\ &+ J^{\textup{v2g}}_{i}(u_i(t))(1-\delta^{\textup{c}}_{i}(t)),
\end{aligned}
\end{equation}
 where $\delta^{\textup{c}}_{i} \coloneqq \col((\delta^{\textup{c}}_{i}(t))_{t\in\mc{T}})$ and $\delta^{\textup{d}}_{i} \coloneqq \col((\delta^{\textup{d}}_{i}(t))_{t\in\mc{T}})$. 


\subsection{Local and global constraints} \label{subsec:constr_on_power_exchange} 
By considering a generic \gls{PEV} $i\in\mc I$,  first we assume that, for all $t \in \mc{T}$, it is desired that the corresponding level of \gls{SoC} is at least equal to a reference level  $x^{\textup{ref}}_i(t)\in[0,1]$, i.e.,
\begin{equation}\label{eq:constraint_desiredlevel}
x^{\textup{ref}}_i(t)\leq x_i(t)\,.
\end{equation}

The driving pattern of each \gls{PEV} limits the possible values of $\delta_i(t)$. Namely, if $\mu_i(t)>0$ then feasible set of $\delta_i(t)$ is $\mc B(t) = 0$; {otherwise} $\mc B(t)= \B$. Thus, this condition translate in a time varying feasible set $\mc B(t)$ of $\delta_i(t)$.

If $\delta_i(t)$ is left free, it may lead  a vehicle to persistently switch between being connected and unconnected to the grid among consecutive time intervals. This behaviour is not only unnatural, {but could also lead to a fast degradation of the battery} \cite{Millner:2010:Battery_degradation}. To exclude this potential source of damage, we impose that, if a \gls{PEV} is plugged-in at a certain time $t\in\mc T$, i.e., $\delta_{i}(t-1)=0$ and $\delta_{i}(t)=1$, then it has to remain connected to the charging station for at least $h_i\leq T, h_i \in \N$ consecutive time intervals. Thus, for all $i \in \mc{I}$ and $t\in\mc T$, this constraint can be rephrased via the logical implication
\begin{equation} \label{eq:logic_consecutiveintervals}
[\delta_{i}(t-1)=0] \wedge [\delta_{i}(t)=1] \Rightarrow [\delta_{i}(t+h) = 1, \, \forall h \leq h_i].
\end{equation} 
Next, to exclude that a \gls{PEV} charges or discharges when it is plugged-out (i.e., $\delta_i(t)=0$), we impose that $
	u_i(t)\in \mc U(t)\coloneqq[\underline{u} \delta_i(t),\overline{u}\delta_i(t)]
$. Thus, $u_i(t)=0$ if $\delta_i(t)=0$, otherwise $u_i$ takes values in $[\underline{ u}, \overline{u}]$.
Moreover, to exclude that a vehicle is plugged-in without being neither charged nor discharged, we impose that  $\delta_i(t)=0$ if $u_i(t)=0$. Thus, the next logical implication needs to be satisfied for all $t \in \mc{T}$ and $i \in \mc{I}$:
\begin{equation}\label{eq:unplug_force}
	[u_i(t) = 0] \iff [\delta_i(t) = 0].
\end{equation}
Successively, due to the intrinsic limitations of the grid capacity $\overline{d}>0$, we assume that the amount of energy required in a single time period by both the \glspl{PEV} and non-\gls{PEV} loads cannot be greater than $\overline{d}$. 
This translates into a constraint on the \glspl{PEV} total demand, i.e.,
\begin{equation}\label{eq:constraint_totalamount}
	d(t) + \textstyle\sum_{j\in\mc{I}} u_j(t) \in \left[0, \overline{d}\right].
\end{equation}



Finally, it is natural to assume that the charging station has a finite number of charging spots, i.e., the number of vehicles that can charge/discharge at the same time cannot exceed a finite value $\overline{v} \in\N$. Thus, we impose, for all $t \in \mc{T}$,
\begin{equation}\label{eq:constraint_charging}
	\textstyle\sum_{j\in\mc{I}} \delta_{j}(t) \leq \overline{v}\:.
\end{equation}


\subsection{Mixed-integer game formulation}
We conclude this section by summarizing a preliminary formulation of the optimization problem for the scheduling and charging/discharging of each \gls{PEV} as follows:

\begin{equation}\label{eq:MPC_first}
\forall i \in \mc{I} : \left\{\begin{aligned}
&\underset{u_i, x_i, \delta_i, \delta^{\textup{d}}_{i}, \delta^{\textup{c}}_{i}}{\textrm{min}} & & J_i(u_i, \delta^{\textup{d}}_{i}, \delta^{\textup{c}}_{i}, \bs{u}_{-i})\\
&\hspace{.65cm}\textrm{s.t.} & & \eqref{eq:SoC_update_u}, x_i(t) \in [0, 1],\\
&&& u_i(t) \in \mc{U}(t), \delta_i(t)\in\mc B(t),\\
&&&\delta^{\textup{c}}_{i}(t), \delta^{\textup{d}}_{i}(t) \in \mathbb{B},\\ 
&&& \eqref{eq:logical_implication_deltar}, \eqref{eq:logical_implication_sigma},\,\text{\eqref{eq:constraint_desiredlevel}--\eqref{eq:constraint_charging}}, \; \forall t \in \mc{T}.
\end{aligned}\right.\tag{$\mc{G}$}
\end{equation} 

We emphasize that several constraints in \eqref{eq:MPC_first}, as well as the cost function, are either expressed as logical implications, or directly depend on the evaluation of a logical proposition. In the next section we translate all these logical implications into mixed-integer-linear  inequalities, hence the problem in \eqref{eq:MPC_first} as a (parametric) mixed integer quadratic problem.

%
%
%

\section{Translating the logical implications into mixed-integer linear constraints}\label{sec:MIL}
\label{sec:logic_to_mix_integer}
%

For the sake of clarity, we adopt the same notation used in \cite{fabiani2018mld,fabiani2018mvad}, i.e.,  we define several patterns of inequalities that allow to handle all the constraints. Specifically, given a linear function $f:\R \rightarrow \R$, let us define $M \coloneqq \textrm{max}_{x \in \mc{X}} f(x)$, $m \coloneqq \textrm{min}_{x \in \mc{X}} f(x)$ with $\mc{X}$ being a  compact set. Then, with $c \in \R$ and $\delta \in \mathbb{B}$, a first system $\mc{S}_{\geq}$ of mixed-integer inequalities correspond to $[\delta = 1] \iff [f(x) \geq c]$, i.e.,
\begin{equation*}
\mc{S}_{\geq}(\delta, f(x), c) \coloneqq \left\{
\begin{aligned}
&(c - m)\delta \leq f(x) - m\\
&(M - c + \epsilon) \delta \geq f(x) - c + \epsilon,
\end{aligned}
\right.
\end{equation*}
Here $\epsilon > 0$ is a small tolerance beyond which the constraint is regarded as violated. With the same idea, the following set of inequalities corresponds to $[\delta = 1] \iff [f(x) \leq c]$:
\begin{equation*}
\mc{S}_{\leq}(\delta, f(x), c) \coloneqq \left\{
\begin{aligned}
&(M - c) \delta \leq M - f(x)\\
&(c + \epsilon - m) \delta \geq \epsilon + c - f(x).
\end{aligned}
\right.
\end{equation*}

Successively, we define the next block of inequalities, involving binary variables only, which allow to solve propositions with logical AND, i.e., $[\delta = 1] \iff [\sigma = 1] \wedge [\gamma = 1]$:
\begin{equation*}
\mc{S}_{\wedge}(\delta, \sigma, \gamma) \coloneqq \left\{
\begin{aligned}
- &\sigma + \delta \leq 0\\
- &\gamma + \delta \leq 0\\
&\sigma + \gamma - \delta \leq 1,
\end{aligned}
\right.
\end{equation*}

Finally, to recast bilinear terms, i.e., products of binary and continuous variables, into a mixed-integer linear formulation, we introduce the following pattern of inequalities.
\begin{equation*}
\mc{S}_{\Rightarrow}(g, f(x), \delta) \coloneqq \left\{
\begin{aligned}
& m \delta \leq g \leq M \delta\\
& - M(1-\delta) \leq	g - f(x) \leq - m(1 - \delta)\\
\end{aligned}\right.
\end{equation*}
The latter is equivalent to: $[\delta = 0] \implies [g = 0]$, while $[\delta = 1] \implies [g = f(x)]$.

\subsection{The mixed-integer-linear constraints}
The logical implications in \eqref{eq:logical_implication_deltar} and \eqref{eq:logical_implication_sigma} translate into:
\begin{align}
&\mc{S}_{\geq}(\delta^{\textup{c}}_{i}(t), u_i(t), 0),\label{eq:first_MI}\\
&\mc{S}_{\leq}(\delta^{\textup{d}}_{i}(t), u_i(t), 0). \label{eq:sigma_u}
\end{align}

By referring to the \gls{SoC} dynamics in \eqref{eq:SoC_update_u}, the bilinear term $u_i \delta_i$ can be handled by following the procedure in \cite{bemporad1999control}, which allows to rewrite it as mixed-integer-linear inequalities by means of additional auxiliary variables (both real and binary, \cite{williams2013model}). Specifically, we define the auxiliary variable $f_{i} \coloneqq u_i \delta_i \in \R$ that shall satisfy the pattern of inequalities:
\begin{equation}\label{eq:second_MI}
\mc{S}_{\Rightarrow}(f_i(t), u_i(t), \delta_i(t)).
\end{equation}
Thus,  the \gls{SoC} dynamics at the generic time $t \in \mc{T}$ reads as
\begin{align}\label{eq:new_SOC}
x_i(t+1) = x_i(t) + b_i f_i(t) - (1-\delta_i(t)) \mu_i(t),
\end{align}

The same approach allows to manage the bilinear terms 
that appear in the cost function \eqref{eq:cost_function}.
Hence, by defining $g_{i}(t) \coloneqq u_i(t)(1- \delta^{\textup{c}}_{i}(t)) \in \R$, $s_{i}(t) \coloneqq u_i(t-1)(1- \delta^{\textup{c}}_{i}(t)) \in \R$, $\ell_{i}(t) \coloneqq u_i(t)(1- \delta^{\textup{d}}_{i}(t)) \in \R$ and $\kappa_{i}(t) \coloneqq u_i(t-1)(1- \delta^{\textup{d}}_{i}(t)) \in \R$, subjected to inequalities, for all $i \in \mc{I}$ and $t \in \mc{T}$,
\begin{align}
	&\mc{S}_{\Rightarrow}(g_i(t), u_i(t), 1-\delta^{\textup{c}}_{i}(t)),\\
	&\mc{S}_{\Rightarrow}(s_i(t), u_i(t-1), 1-\delta^{\textup{c}}_{i}(t)),\\
	&\mc{S}_{\Rightarrow}(\ell_i(t), u_i(t), 1-\delta^{\textup{d}}_{i}(t)),\\
	&\mc{S}_{\Rightarrow}(\kappa_i(t), u_i(t-1), 1-\delta^{\textup{d}}_{i}(t)) \label{eq:third_MI},
\end{align} 
the cost function \eqref{eq:cost_function} reads as (omitting the dependencies),
\begin{align}\label{eq:new_cost_function}
	J_i = c &(d(t)+a_i(\bs{u}(t))) \ell_i(t) + r_i(t) g_i(t) \nonumber\\  
	& + \rho_i^{-} \left(g_i(t) -s_i(t)\right)^2 + \rho_i^{+} (\ell_i(t) - \kappa_i(t) )^2 \:. 
\end{align}
Notice that via these variables we can rewrite also $a(\bs{u}(t))=\sum_{j\in\mc I\setminus\{i\}} \ell_j(t)$.
Let us consider proposition \eqref{eq:unplug_force}, equivalently rewritten as
$
[u_i(t) \leq 0] \wedge [u_i(t) \geq 0] \implies [(1-\delta_i(t)) = 1].
$
Since $\delta^{\textup{c}}_{i}$ and $\delta^{\textup{d}}_{i}$ already satisfy the inequalities in \eqref{eq:first_MI} and \eqref{eq:sigma_u}, for all $t \in \mc{T}$, the constraint in \eqref{eq:unplug_force} reads as
\begin{equation}\label{eq:bau}
	\mc{S}_{\wedge}(1-\delta_i(t), \delta^{\textup{c}}_{i}(t), \delta^{\textup{d}}_{i}(t)).
\end{equation}

Next, by referring to \eqref{eq:logic_consecutiveintervals}, we introduce a binary auxiliary variable $\alpha_i$, which is equal to $1$ when both $\delta_i(t-1)=0$ and $\delta_i(t)=1$, $0$ otherwise. Hence, $\alpha_i$ shall satisfy, for all $t \in \mc{T}$,
$
	[\alpha_i(t) = 1] \iff [(1-\delta_i(t-1))=1] \wedge [\delta_i(t)=1],
$
which translates into the pattern of integer linear inequalities
\begin{equation}\label{eq:fourth_MI}
	\mc{S}_{\wedge}(\alpha_i(t), 1-\delta_i(t-1), \delta_i(t)).
\end{equation}
Thus, \eqref{eq:logic_consecutiveintervals} can be rewritten as a nonlinear equality constraint
$
\alpha_i(t) \left(\sum_{h =1}^{h_i} \delta_i(t+h) - h_i \right) = 0.
$
Successively, we introduce a set of variables $\beta^{(h)}_i \in \B$ such that, for all $t \in \mc{T}$,
$
[\beta^{(h)}_i(t) = 1] \iff [\alpha_i(t)=1] \wedge [\delta_i(t+h)=1, \, \forall h \leq h_i],
$
which corresponds to the following patterns of inequalities
\begin{equation}\label{eq:fifth_MI}
\mc{S}_{\wedge}(\beta^{(h)}_i(t), \alpha_i(t), \delta_i(t+h)),  \; \forall h \leq h_i.
\end{equation}
This allows to rewrite the equality constraint in linear form
\begin{equation}\label{eq:sixth_MI}
	\textstyle{\sum_{h = 1}^{h_i} }\beta^{(h)}_i(t) - \alpha_i(t) h_i = 0.
\end{equation}

\subsection{Final mixed-integer-linear model}
By rearranging all the inequalities, we obtain the following mixed-integer aggregative game 
\begin{equation}\label{eq:MPC_i_complete}
\forall i \in \mc{I} : \left\{\begin{aligned}
&\underset{u_i, x_i, \ldots,\beta_i}{\textrm{min}} & & J_i(u_i, g_i, s_i, \ell_i, \kappa_i, \bs{u}_{-i})\\
&\hspace{.5cm}\textrm{s.t.} & &  \eqref{eq:constraint_desiredlevel},  x_i(t) \in [0, 1],\\
&&& \text{\eqref{eq:constraint_totalamount}--\eqref{eq:third_MI}}, \text{\eqref{eq:bau}--\eqref{eq:sixth_MI}},\\
&&& u_i(t), f_i(t), g_i(t), s_i(t) \in \mc{U}(t),\\
&&& \ell_i(t), \kappa_i(t) \in \mc{U}(t),\;\delta_i(t)\in\mc B(t),\\ 
&&&\delta^{\textup{c}}_{i}(t), \alpha_i(t), \delta^{\textup{d}}_{i}(t) \in \mathbb{B}, \; \forall t \in \mc{T},\\
&&& \beta^{(h)}_i(t) \in \mathbb{B}, \;\forall h \leq h_i,  \forall t \in \mc{T}.
\end{aligned}\right.
\end{equation}

The coupling constraints in \eqref{eq:constraint_totalamount}--\eqref{eq:constraint_charging}, as well as the cost function, depend on the aggregate behavior of the fleet of \glspl{PEV}. By defining $z_i \coloneqq \col(u_i,x_i,\ldots,\beta_i) \in \R^{n_i}$ and $\bs{z} \coloneqq \col((z_i)_{i \in \mc{I}}) \in \R^{n}$, $n \coloneqq \sum_{i \in \mc{I}} n_i$, as the vectors stacking local and collective mixed-integer variables, respectively, we have, for some suitable matrix $A$ and vector $b$,
\begin{equation}\label{eq:MPC_i_compact}
\forall i \in \mc{I} : 
\underset{z_i}{\textrm{min}} \; J_i(z_i, \bs{z}_{-i}) \;\; \textrm{s.t. } \; A \bs{z} \leq b.
\end{equation}

\section{PEVs charge coordination as a generalized mixed-integer potential game}
\label{sec:gen_mixed_int_pot_game}
Our aim is now to design suitable sequences of decisions that guarantee to each \gls{PEV} an effective, economic, possibly profitable, usage of the electrical distribution network, while satisfying both the intrinsic limitations of the charging station itself, and by the presence of the other vehicles. 
We propose to achieve such a trade-off by formalizing the charge coordination of the set of \glspl{PEV} as a \gls{MI-GPG} \cite{facchinei2011decomposition}, an instance of the \glspl{GNEP} \cite{facchinei2007generalized}.

\subsection{Potential game setup}
First, we identify the player set with $\mc{I}$, and we define the feasible set of each player, i.e., $\mc{Z}_{i}(\bs{z}_{-i}) \coloneqq \{ z_i \in \R^{n_i} \mid A (z_i, \bs{z}_{-i}) \leq b \}$, and $\bs{\mc{Z}} \coloneqq \{\bs{z} \in \R^n \mid A \bs{z} \leq b\}$. Next, we introduce the mixed-integer best response mapping for player $i$, given the strategies of the other players $\bs{z}_{-i}$:
\begin{equation}\label{eq:bestresponsemapping}
z^{\star}_{i}(\bs{z}_{-i}) \coloneqq 
\underset{z_i}{\textrm{argmin}} \; J_i(z_i, \bs{z}_{-i}) \;\; \textrm{s.t. } \; (z_i, \bs{z}_{-i}) \in \bs{\mc{Z}}.
\end{equation}

Hence, in the proposed game $\Gamma \coloneqq \left(\mc{I}, \{J_i\}_{i \in \mc{I}}, \{\mc{Z}_i\}_{i \in \mc{I}} \right)$, we are interested in the following notion of equilibrium.
\begin{definition}[$\varepsilon$-Mixed-Integer Nash equilibrium]\label{def:MINE}
	Let $\varepsilon > 0$. $\bs{z}^{*} \in \bs{\mc{Z}}$ is an \gls{MINE} of the game $\Gamma$ in \eqref{eq:bestresponsemapping} if, for all $i \in \mc{I}$,
	\begin{equation*}
	J_i(z^{*}_i, \bs{z}^{*}_{-i}) \leq \underset{z_i \in \mc{Z}_i(\bs{z}^{*}_{-i})}{\textrm{inf}} J_i(z_i, \bs{z}^{*}_{-i}) + \varepsilon.
	\end{equation*}
	
	\vspace{-.88cm}
	\hfill$\square$
	\vspace{.15cm}
\end{definition}

We consider here potential games characterized by the existence of an exact potential function, defined next.

\begin{definition}
	A continuous function $P:\R^n \to \R$ is an exact potential function for the game $\Gamma$ in \eqref{eq:bestresponsemapping} if, for all $i \in \mc{I}$, for all $\bs{z}_{-i}$, and for all  $z_i$, $y_i \in \mc{Z}_i(\bs{z}_{-i})$,
	\begin{equation*}
	P(z_i,\bs{z}_{-i}) - P(y_i,\bs{z}_{-i}) = J_i(z_i,\bs{z}_{-i}) - J_i(y_i,\bs{z}_{-i}). \quad\square
	\end{equation*}
%
\end{definition}

Now, let us consider the cost function in \eqref{eq:new_cost_function} defined for all $t \in \mc{T}$. Due to the particular structure of the cost of the acquired energy, in compact form it reads as
\begin{equation}\label{eq:cost_fun_compact}
J_i(z_i, \bs{z}_{-i}) = \phi_i(z_i) + \textstyle\sum_{j\in\mc{I}\setminus\{i\}} \omega_{ij} (z_i, z_{j}),
\end{equation}
where $\phi_i(z_i) \coloneqq z_i^\top Q_i z_i + q_i^\top z_i$, for some suitable positive semi-definite matrix $Q_i \in \R^{n_i \times n_i}$ and vector $q_i \in \R^{n_i }$, while 
$$
	\omega_{ij} (z_i, z_{j}) \coloneqq z_j^\top \left[\begin{smallmatrix}
	 \bs{0} & \bs{0} & \bs{0}\\
	 \bs{0} & c \, I & \bs{0}\\
	 \bs{0} & \bs{0} & \bs{0}
	\end{smallmatrix}\right]  z_i = c \,\ell_j^\top \ell_i.
$$
\begin{theorem}\label{th:pot_game}
	The game $\Gamma$ in \eqref{eq:bestresponsemapping} is an \gls{MI-GPG} with exact potential function
	\begin{equation}\label{eq:pot_game}
		P(\bs{z}) \coloneqq \textstyle\sum_{i \in \mc{I}}\left(\phi_i(z_i) + \textstyle\sum\limits_{j \in \mc{I}, j < i} \omega_{ij} (z_i, z_j)\right).
	\end{equation}
	\hfill$\square$
\end{theorem}
\begin{proof}
	The proof replicates the one in \cite[Prop.~2]{fabiani2019nash}.
\end{proof}

To conclude, we recall that the set of $\varepsilon$-approximated minimum over $\bs{\mc{Z}}$ of $P$ corresponds to a subset of the \gls{MINE} of the game \cite[Th.~2]{sagratella2017algorithms}.

\subsection{Semi-decentralized solution algorithm}
\setlength{\algomargin}{.5em}
\begin{algorithm}[!t]
	\caption{Sequential best-response}\label{alg:Alg_GS}
	\DontPrintSemicolon
	\SetArgSty{}
	\SetKwIF{If}{ElseIf}{Else}{if}{}{else if}{else}{end if}
	\SetKwFor{ForAll}{for all}{do}{end forall}
	\SetKwRepeat{Do}{do}{end}
	\textbf{Initialization: }Choose $\bs{z}(0) \in \bs{\mc{Z}}$, set $k \coloneqq 0$\;
	\While{$\bs{z}(k)$  is not an \gls{MINE}}{
		$\mc{AG}$ \Do{}{
			Chooses $i \coloneqq i(k) \in \mathcal{I}$\\
			Sets $p_i(\bs{u}(k))$ as in \eqref{eq:price_of_energy}, $e(\bs{\delta}(k))$ as in \eqref{eq:free_plugs}\\ 
			Sends $p_i(\bs{u}(k))$, $e(\bs{\delta}(k))$ to $i$
			
		}
		Player $i$ \Do{}{
			Compute $z^{\ast}_{i}(k) \in z^{\star}_{i}(k)$ as in \eqref{eq:bestresponsemapping}\\
			 \uIf{$J_{i}(z_{i}(k), p_i(\bs{u}(k))) - J_{i}(z_{i}^{\ast}(k), p_i(\bs{u}(k))) \geq \varepsilon$}{ \smallskip $z_{i}(k+1) \coloneqq z^{\ast}_i(k)$ \smallskip}		
			 \Else{ $z_{i}(k+1) \coloneqq z_i(k)$ \smallskip}
		}
		$\mc{AG}$ collects $z_{i}(k+1)$\\
		Set $z_j(k+1) \coloneqq z_j(k) \; \forall j\neq i$, $k \coloneqq k+1$
}
\end{algorithm}

\begin{figure}[b]
	\centering
	\includegraphics[trim=120 90 150 20,clip,width=0.74\columnwidth]{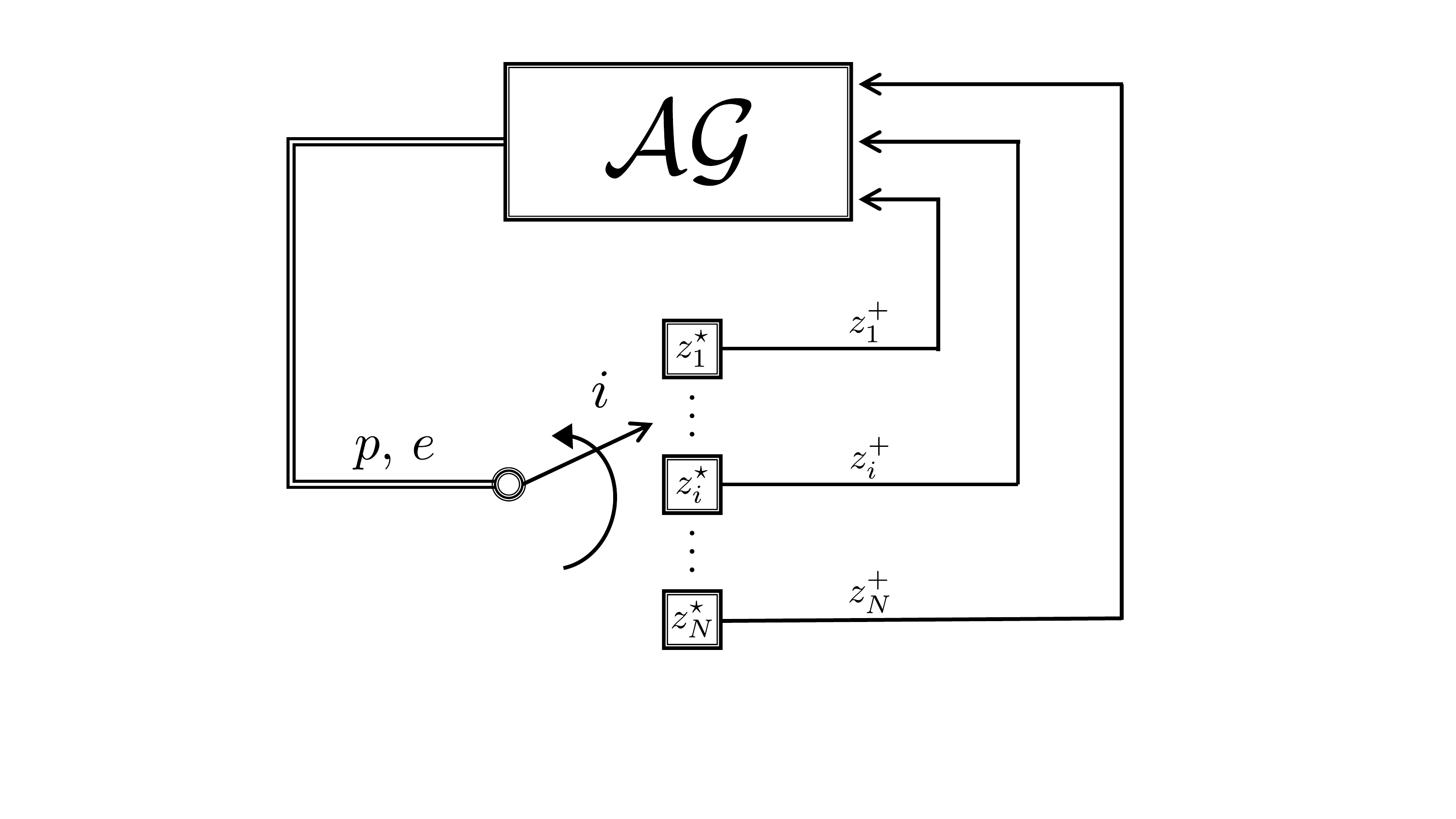}
	\caption{Semi-decentralized implementation of Algorithm~\ref{alg:Alg_GS}, where $z_i^+\coloneqq z_i(k+1)$.}
	\label{fig:agg_scheme}
\end{figure}

To compute a solution to \eqref{eq:MPC_i_compact}, each player requires the aggregate information over the whole horizon $\mc{T}$, i.e., the \glspl{PEV} demand of energy $a_i(\bs{u}) \coloneqq \col((a_i(\bs{u}(t)))_{t\in\mc{T}})$, or, equivalently, the price associated to such demand
\begin{equation}\label{eq:price_of_energy}
	p_i(\bs{u}) \coloneqq c (d + a_i(\bs{u})),
\end{equation}
with $d \coloneqq \col((d(t))_{t \in \mc{T}})$, and the number of available charging points at the charging station, i.e., 
\begin{equation}\label{eq:free_plugs}
	e(\bs{\delta}) \coloneqq \overline{v}\bs{1}_{T} -  \left( \bs{1}_{N}^\top \otimes I_T \right) \bs{\delta},
\end{equation}
where $\bs{\delta} \coloneqq \col((\delta_{i})_{i \in \mc{I}})$. Specifically, while the number of available charging points $e(\bs{\delta})$ reflects the coupling constraint in \eqref{eq:constraint_charging}, the price of energy $p_i(\bs{u})$ affects both the constraints and the cost function in \eqref{eq:cost_fun_compact}, 
which can be equivalently expressed as function of $z_i$ and $p_i(\bs{u})$.

In this context, it seems unrealistic or, at least, impractical to assume full communication among all the potential users of the same charging station, over the same horizon $\mc{T}$. Therefore, to compute an \gls{MINE}, we propose a semi-decentralized implementation of the sequential best-response method, summarized in Algorithm~\ref{alg:Alg_GS}, that includes an aggregator $\mc{AG}$, whose role is played by the charging station, within the communication pattern. 
According to \cite{Review_2016}, we assume the  driving patterns of each \gls{PEV} to be known  by the $\mc{AG}$.
As shown in the scheme in Fig.~\ref{fig:agg_scheme},  at each iteration $k$, the aggregator $\mc{AG}$ chooses the next agent taking part in the game, communicating the price of energy and the number of available charging points. With this information, the selected agent computes a best response and decides to change strategy only if it leads to an (at least) $\varepsilon$-improvement in terms of minimization of its cost function. Finally, $\mc{AG}$ collects the decision of such agent, while the remaining part of players keeps its strategy unchanged.
%

\begin{corollary}
	Let $\varepsilon > 0$ and assume that, in Algorithm~\ref{alg:Alg_GS}, there exists $K > 0$ such that $j \in \{i(k), i(k+1),\ldots,i(k+K)\}$ for all $j \in \mc{I}$ and $k \geq 0$. Algorithm~\ref{alg:Alg_GS} computes an \gls{MINE}, $\bs{z}^\ast \in \bs{\mc{Z}}$, of the game $\Gamma$ in \eqref{eq:bestresponsemapping}.
	\hfill$\square$
\end{corollary}
\begin{proof}
	In view of Theorem~\ref{th:pot_game}, the game $\Gamma$ in \eqref{eq:bestresponsemapping} is an \gls{MI-GPG}. Therefore, by \cite[Th. 4]{sagratella2017algorithms}, the sequential best-response based Algorithm~\ref{alg:Alg_GS} converges to an \gls{MINE}.
\end{proof}

\section{Numerical simulations}
\label{sec:num_sim}
In this section, we show numerical results obtained by solving the \gls{MI-GPG} in \eqref{eq:bestresponsemapping} with parameter values specified in Tab.~\ref{tab:sim_val} (see \cite{J_Hiskens_2013,J_Rivera_2017, Melle}) and reward function $r(t)$ chosen proportional to the nominal demand and the energy generated by the vehicles, i.e., $r(t)\coloneqq \overline{r}(\sum_{j\in\mc I\setminus\{i\}} u^-_j(t)+d(t))$, where $d(t)$ corresponds to a typical daily demand curve \cite{J_Rivera_2017}, and $u^-_j(t)$ defined analogously to  $u^+_j(t)$.
\begin{table}[tb]
\caption{Simulation parameters}
\label{tab:sim_val}
\begin{center}
\begin{tabular}{ |p{0.75cm}p{1cm}| p{3.5cm}p{1.8cm}|  }
 \hline
  Name & [$\:$]&  Description & Values \\
 \hline
 $N$ &  &  \glspl{PEV} number   &$6$  \\
 $T$ & &  Time intervals in $24h$ & $48$ \\
 $u$ &[kWh] & Energy exchanged per interval & $[-7.5,7.5]$ \\
 $\eta$& & Battery efficiency & $0.85$ \\
 $C_i$ &[kWh] & Battery capacity  & $[40, 75]$\\
 $x_0$ && Initial SoC of battery & $0.23$\\	 
 $x^{\mathrm{ref}}$ && Ref. battery SoC  & $[0.2, 0.85]$ \\
 $c$ &[\euro /kWh] & Energy cost & $1.09\times10^{-3}$ \\
 $\bar r$ &[\euro /kWh] & Constant reward & $1.23\times10^{-3}$ \\
 $\rho^+\, (\rho^-)$& [\euro/kWh$^{2}$] & Degradation cost & $1(0.5)\times 10^{-3}$  \\
 $\overline{d}$ & [kWh] & Grid power capacity & $45$\\
 $ h $ & & Min. consecutive v2g slots & $5$\\
 $\overline{v}$ && Max. connected \glspl{PEV} & $5$ \\
 \hline
\end{tabular}
\end{center}
\end{table} 

For a subset of \glspl{PEV}, we show in Fig.~\ref{fig:x} that the \gls{SoC} $x_i(t)$, as well as the scheduling variable $\delta_{i}(t)$, satisfy the constraints on the required \gls{SoC} $x_i^{\mathrm{ref}}(t)$ \eqref{eq:constraint_desiredlevel} and on the number $h_i$ of plugged-in consecutive intervals \eqref{eq:logic_consecutiveintervals}.
\begin{figure*}[t]
    \centering
    \begin{subfigure}[b]{0.32\textwidth}
        \includegraphics[width=\textwidth]{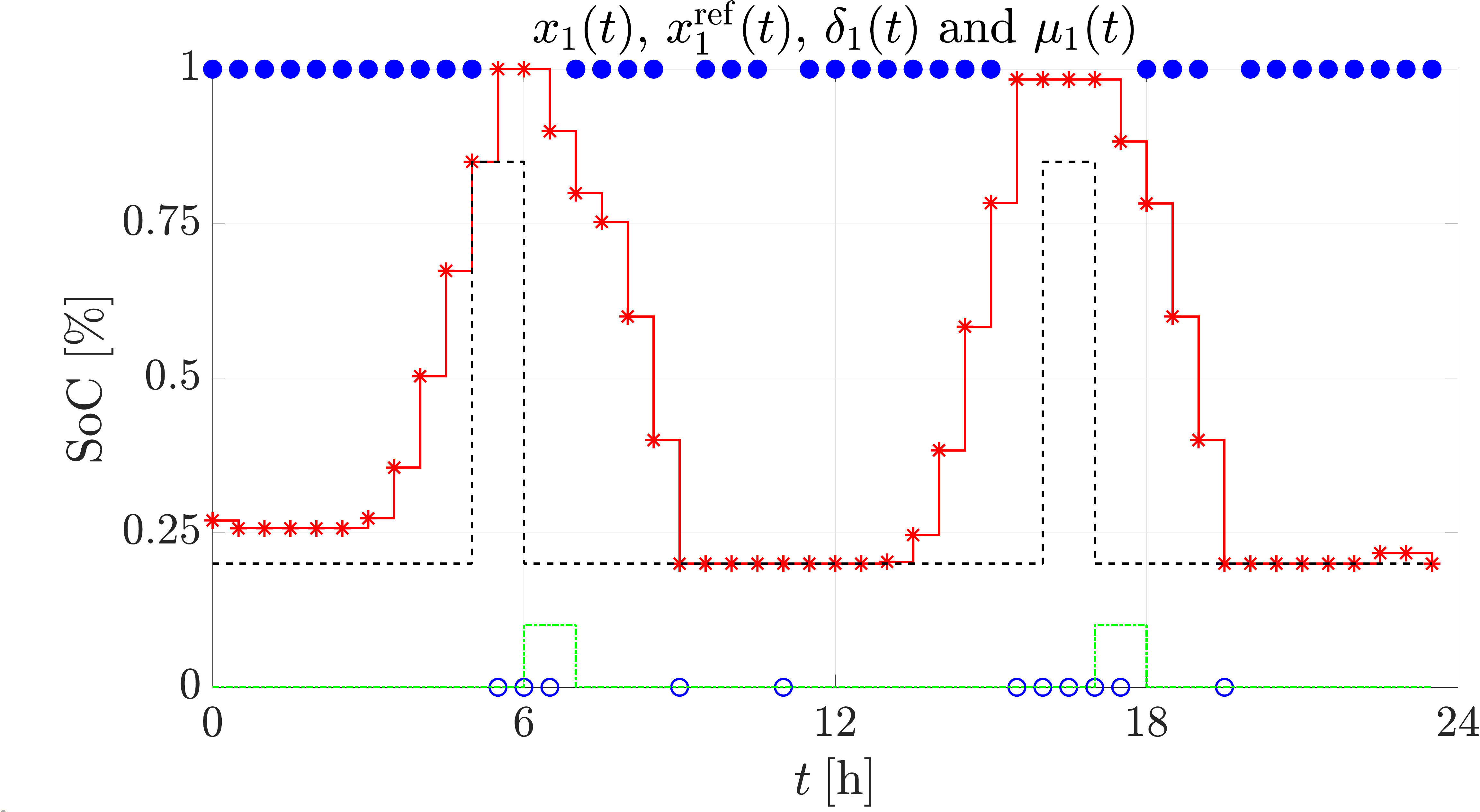}
        \caption{}
        \label{fig:x_1}
    \end{subfigure}
    ~ 
    \begin{subfigure}[b]{0.32\textwidth}
        \includegraphics[width=\textwidth]{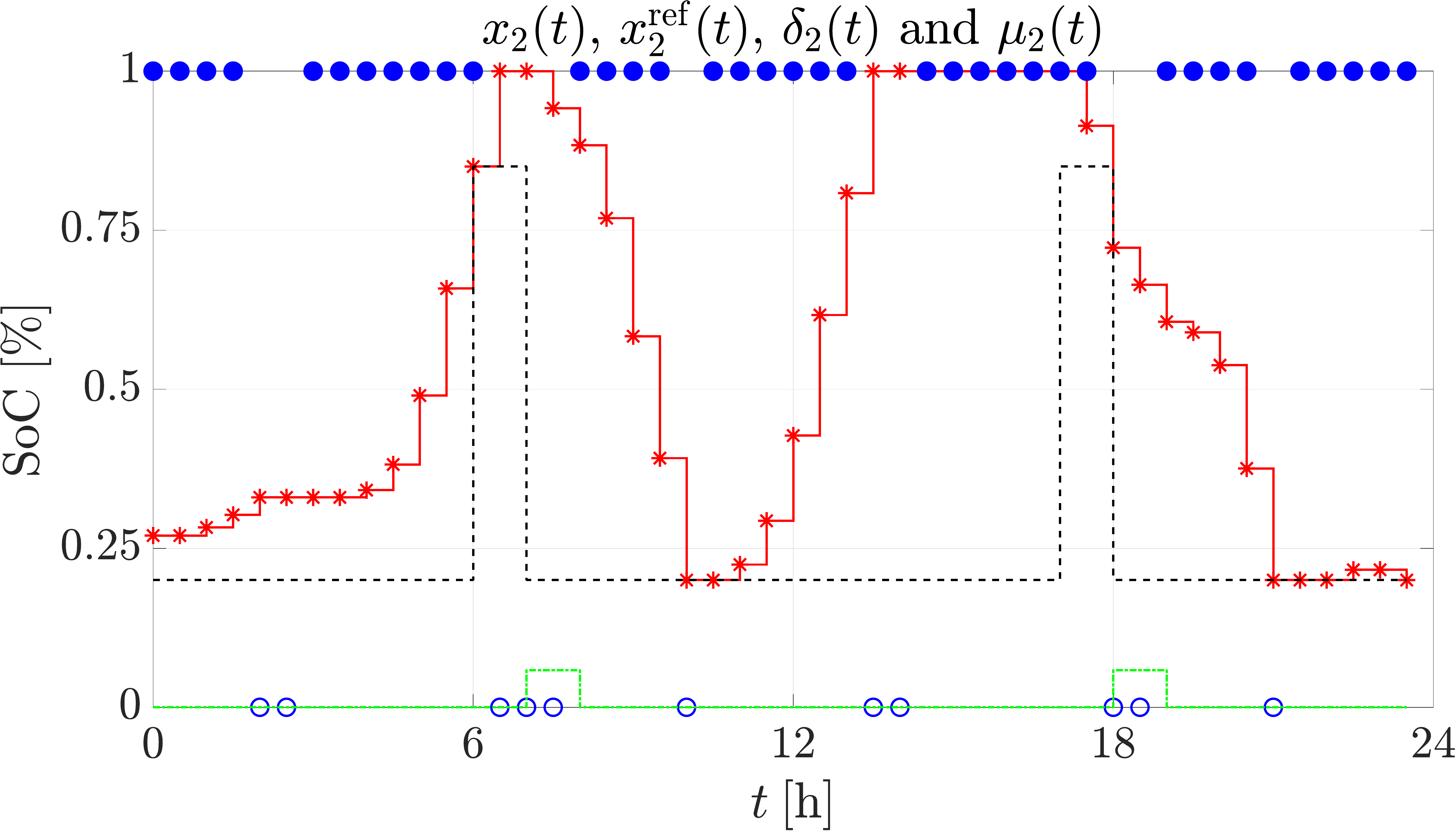}
        \caption{}
        \label{fig:x_2}
    \end{subfigure}
    \begin{subfigure}[b]{0.32\textwidth}
        \includegraphics[width=\textwidth]{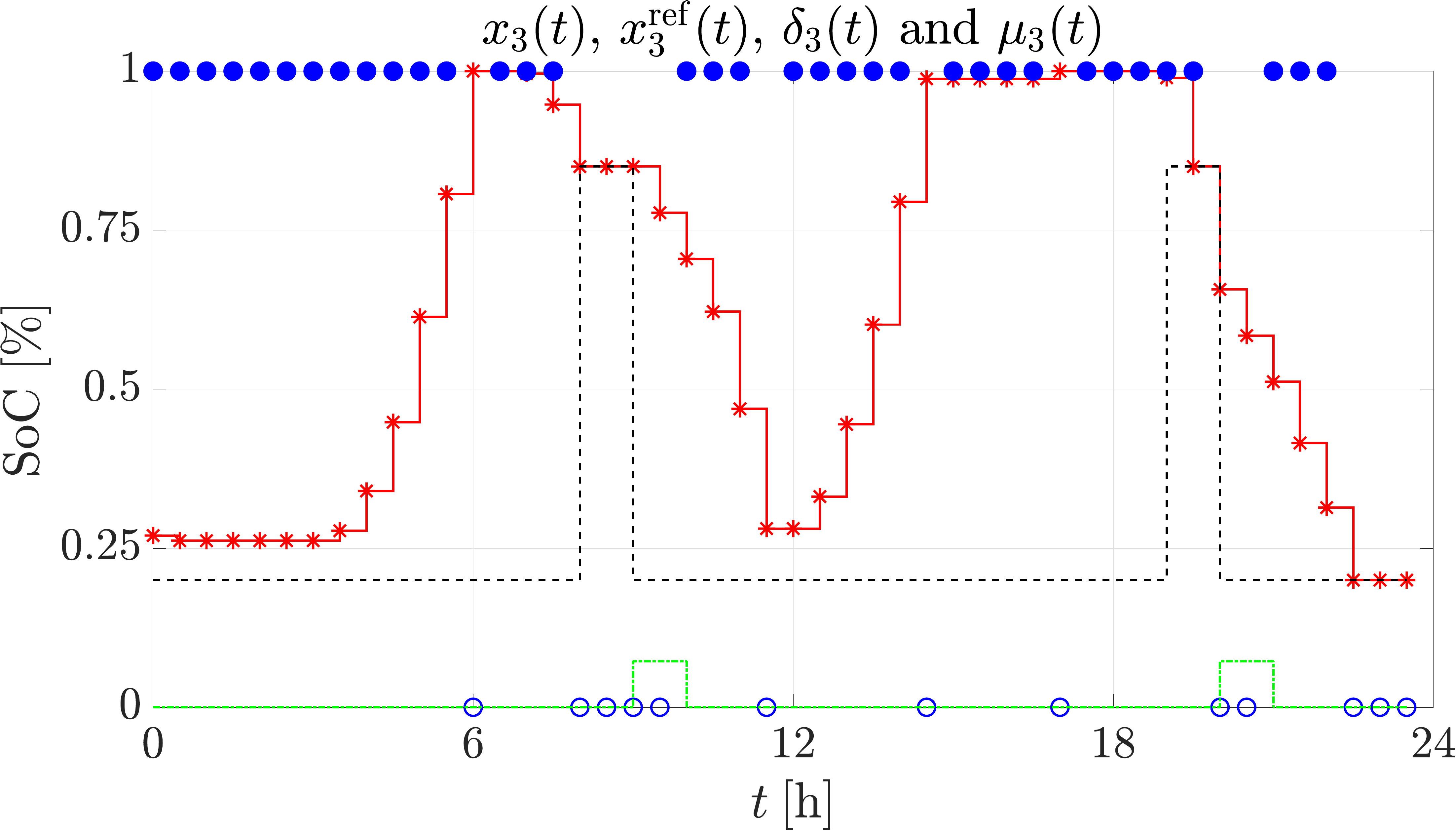}
        \caption{}
        \label{fig:x_3}
    \end{subfigure}\\
        \begin{subfigure}[b]{0.32\textwidth}
        \includegraphics[width=\textwidth]{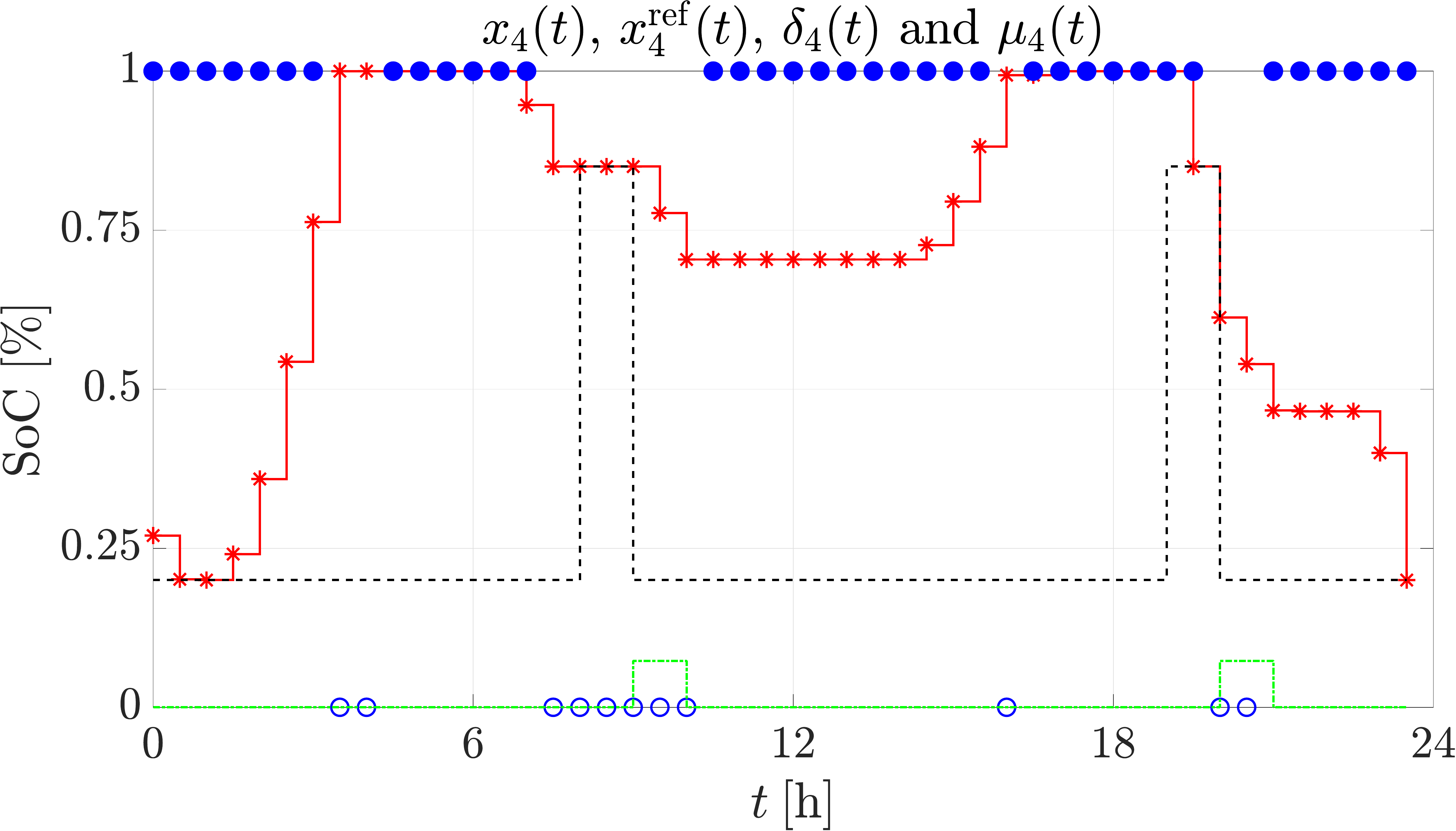}
        \caption{}
        \label{fig:x_4}
    \end{subfigure}
    ~ 
    \begin{subfigure}[b]{0.32\textwidth}
        \includegraphics[width=\textwidth]{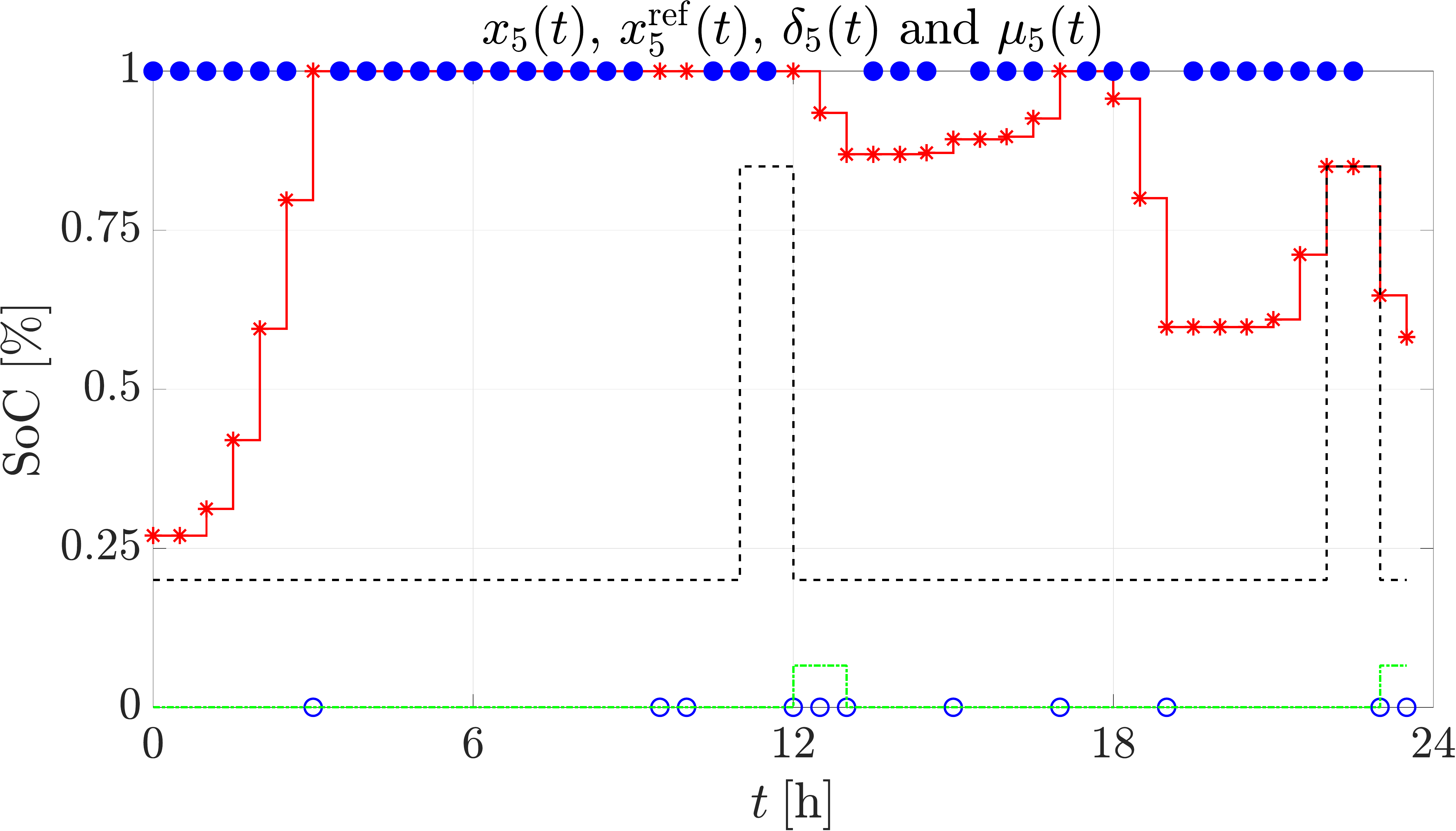}
        \caption{}
        \label{fig:x_5}
    \end{subfigure}
    \begin{subfigure}[b]{0.32\textwidth}
        \includegraphics[width=\textwidth]{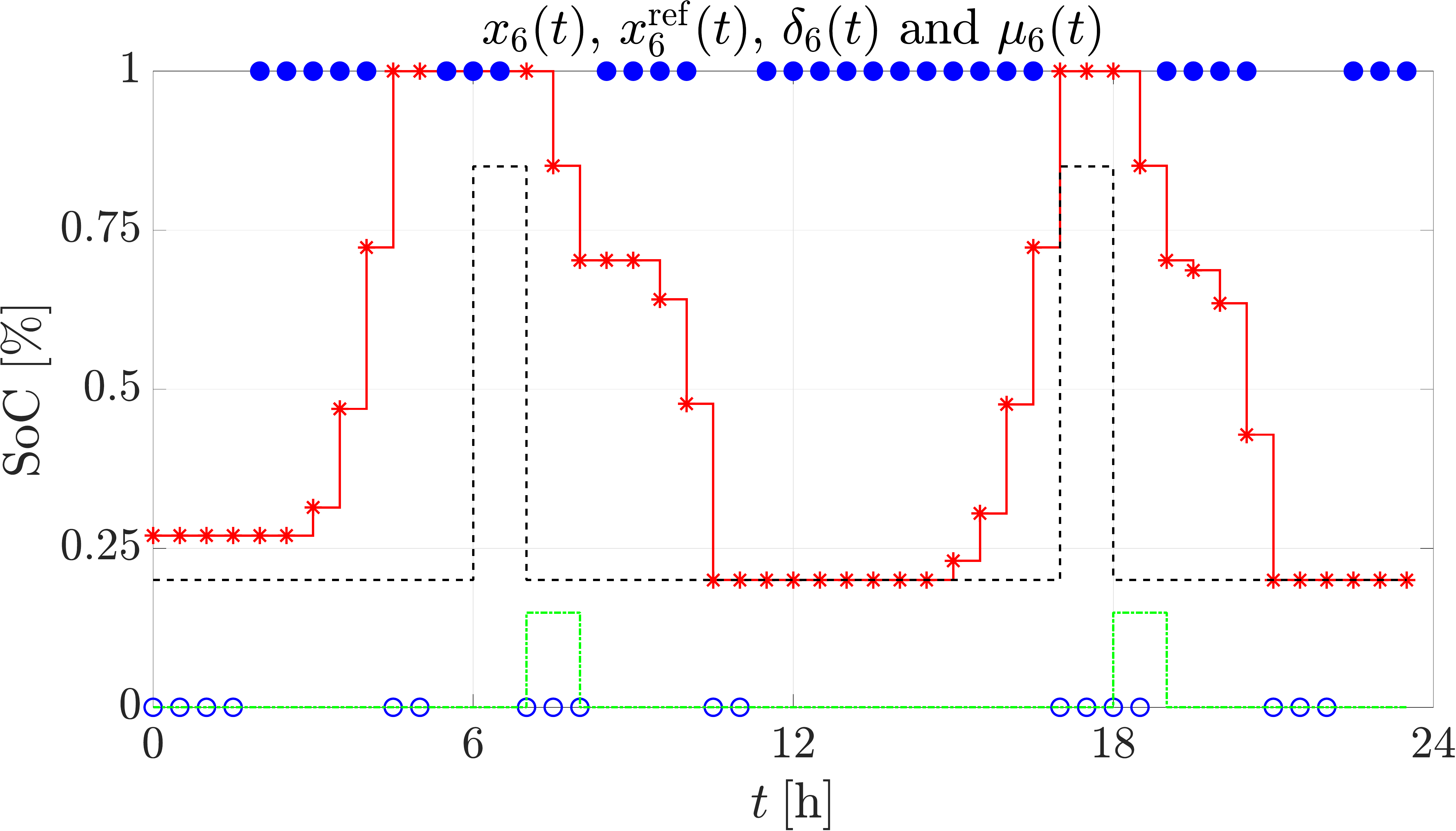}
        \caption{}
        \label{fig:x_6}
    \end{subfigure}
    \caption{For each \gls{PEV} $i\in \mc I$, the SoC $x_i(t)$ (solid red line), the desired lower bound of the SoC $x_i^{\mathrm{ref}}(t)$ (dashed black line), the value of the binary sheduling variable $\delta_i(t)$ (empty and filled blue circles) and $\mu_i(t)$ (dashed green line), over the horizon $\mc{T}$.}\label{fig:x}
\end{figure*}
In Fig.~\ref{fig:a_and_dt}, note that the peak value of  $d(t)$ exceeds the maximum capacity $\overline d$ at $t=$ 19 h. Consequently, the constraint \eqref{eq:constraint_totalamount} forces some of the vehicles to discharge during this time period, showing the  so-called ``valley filling'' phenomenon.
\begin{figure}[tb!]
	\centering
	\includegraphics[width=.8\columnwidth]{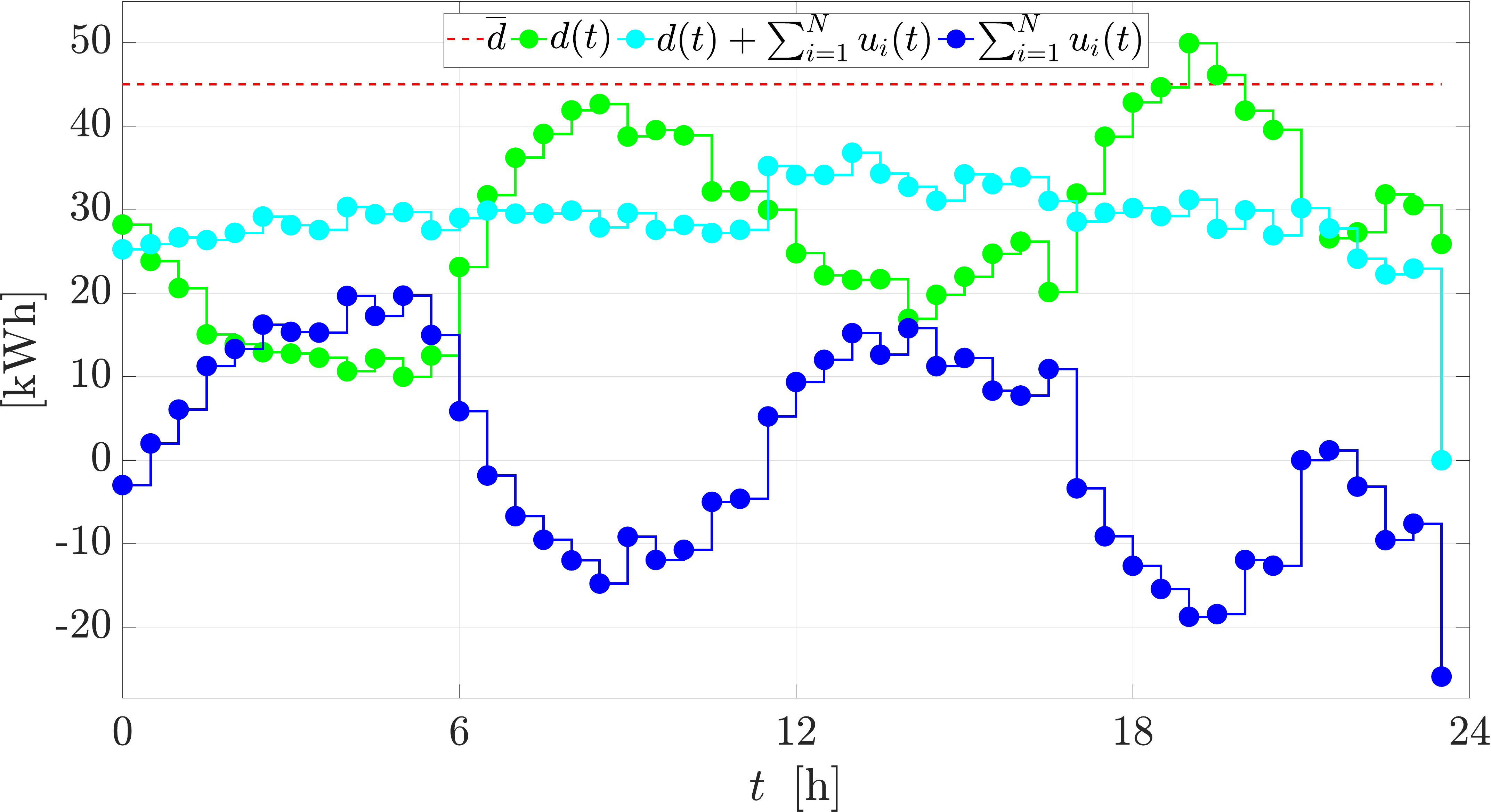}
	\caption{Power required by the load in the grid $d(t)$ (green), the one exchanged with the fleet, $\sum_{i=1}^N u_i(t)$ (blue) and the maximum capacity of the charging station, $\overline d$ (red dashed). }
	\label{fig:a_and_dt}
\end{figure}
Finally, Fig.~\ref{fig:pot_fun} shows the value of the potential function $P(\bld z)$ and the convergence of the algorithm after approximately $25$ iterations.

\begin{figure}[tb!]
	\centering
	\includegraphics[width=.8\columnwidth, trim= 0 0 0 0,clip]{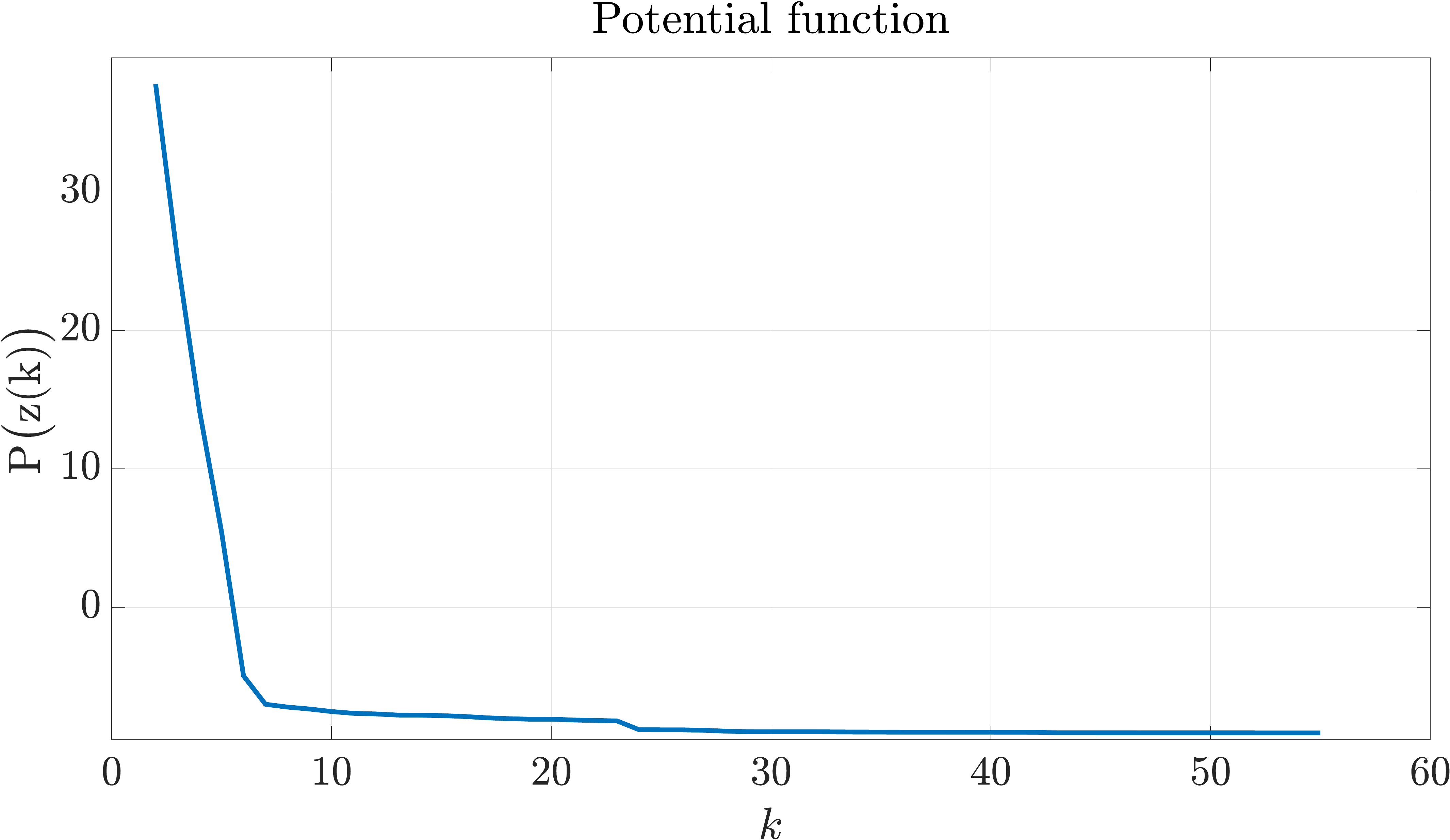}
	\caption{The values of the potential function $P(\bld z)$ associated to the updating \gls{PEV} in Algorithm~\ref{alg:Alg_GS}.}
	\label{fig:pot_fun}
\end{figure}

\section{Conclusion and outlook}
\label{sec:conclusion}
A mixed-integer aggregative game can efficiently solve the charge scheduling coordination of a fleet of \glspl{PEV}, each one modeled as a system of mixed-logical-dynamical systems. Discrete decision variables, as well as logical propositions, allow to catch the different interests of each \gls{PEV} and the heterogeneous interactions between both the set of vehicles itself and the charging station. By playing the role of central aggregator, the charging station is key in computing an (approximated) equilibrium of the associated generalized potential game in a semi-decentralized fashion. In fact, the specific sequence of activation of the \glspl{PEV} highly influences both the outcome of the game and the speed of convergence.
The analysis of this phenomenon is left to future works, as well as the generalization in the case of multiple charging stations, i.e., possible multiple central aggregators.

\balance
\bibliographystyle{IEEEtran}
\bibliography{19_CDC_PEV_MC}

\begin{thebibliography}{10}
\providecommand{\url}[1]{#1}
\csname url@samestyle\endcsname
\providecommand{\newblock}{\relax}
\providecommand{\bibinfo}[2]{#2}
\providecommand{\BIBentrySTDinterwordspacing}{\spaceskip=0pt\relax}
\providecommand{\BIBentryALTinterwordstretchfactor}{4}
\providecommand{\BIBentryALTinterwordspacing}{\spaceskip=\fontdimen2\font plus
\BIBentryALTinterwordstretchfactor\fontdimen3\font minus
  \fontdimen4\font\relax}
\providecommand{\BIBforeignlanguage}[2]{{%
\expandafter\ifx\csname l@#1\endcsname\relax
\typeout{** WARNING: IEEEtran.bst: No hyphenation pattern has been}%
\typeout{** loaded for the language `#1'. Using the pattern for}%
\typeout{** the default language instead.}%
\else
\language=\csname l@#1\endcsname
\fi
#2}}
\providecommand{\BIBdecl}{\relax}
\BIBdecl

\bibitem{Review_2016}
J.~Hu, H.~Morais, T.~Sousa, and M.~Lind, ``Electric vehicle fleet management in
  smart grids: A review of services, optimization and control aspects,''
  \emph{Renewable and Sustainable Energy Reviews}, vol.~56, pp. 1207--1226,
  2016.

\bibitem{doi:10.1002/etep.565}
R.~J. Bessa and M.~A. Matos, ``Economic and technical management of an
  aggregation agent for electric vehicles: A literature survey,''
  \emph{European Transactions on Electrical Power}, vol.~22, no.~3, pp.
  334--350, 2012.

\bibitem{5356176}
K.~{Clement-Nyns}, E.~{Haesen}, and J.~{Driesen}, ``The impact of charging
  plug-in hybrid electric vehicles on a residential distribution grid,''
  \emph{IEEE Transactions on Power Systems}, vol.~25, no.~1, pp. 371--380,
  2010.

\bibitem{6021358}
E.~{Sortomme} and M.~A. {El-Sharkawi}, ``Optimal scheduling of vehicle-to-grid
  energy and ancillary services,'' \emph{IEEE Transactions on Smart Grid},
  vol.~3, no.~1, pp. 351--359, 2012.

\bibitem{J_Hiskens_2013}
Z.~{Ma}, D.~S. {Callaway}, and I.~A. {Hiskens}, ``Decentralized charging
  control of large populations of plug-in electric vehicles,'' \emph{IEEE
  Transactions on Control Systems Technology}, vol.~21, no.~1, pp. 67--78,
  2013.

\bibitem{C_Parise_2014}
F.~{Parise}, M.~{Colombino}, S.~{Grammatico}, and J.~{Lygeros}, ``Mean field
  constrained charging policy for large populations of plug-in electric
  vehicles,'' in \emph{53rd IEEE Conference on Decision and Control}, 2014, pp.
  5101--5106.

\bibitem{Gan_2013}
L.~{Gan}, U.~{Topcu}, and S.~H. {Low}, ``Optimal decentralized protocol for
  electric vehicle charging,'' \emph{IEEE Transactions on Power Systems},
  vol.~28, no.~2, pp. 940--951, 2013.

\bibitem{ma2016efficient}
Z.~Ma, S.~Zou, L.~Ran, X.~Shi, and I.~A. Hiskens, ``Efficient decentralized
  coordination of large-scale plug-in electric vehicle charging,''
  \emph{Automatica}, vol.~69, pp. 35--47, 2016.

\bibitem{Grammatico_2016}
S.~{Grammatico}, ``Exponentially convergent decentralized charging control for
  large populations of plug-in electric vehicles,'' in \emph{2016 IEEE 55th
  Conference on Decision and Control (CDC)}, 2016, pp. 5775--5780.

\bibitem{Zou_2017}
S.~{Zou}, Z.~{Ma}, X.~{Liu}, and I.~{Hiskens}, ``An efficient game for
  coordinating electric vehicle charging,'' \emph{IEEE Transactions on
  Automatic Control}, vol.~62, no.~5, pp. 2374--2389, 2017.

\bibitem{J_You_2016}
P.~{You}, Z.~{Yang}, M.~{Chow}, and Y.~{Sun}, ``Optimal cooperative charging
  strategy for a smart charging station of electric vehicles,'' \emph{IEEE
  Transactions on Power Systems}, vol.~31, no.~4, pp. 2946--2956, 2016.

\bibitem{J_Xing_2016}
H.~{Xing}, M.~{Fu}, Z.~{Lin}, and Y.~{Mou}, ``Decentralized optimal scheduling
  for charging and discharging of plug-in electric vehicles in smart grids,''
  \emph{IEEE Transactions on Power Systems}, vol.~31, no.~5, pp. 4118--4127,
  2016.

\bibitem{J_Rivera_2017}
J.~{Rivera}, C.~{Goebel}, and H.~{Jacobsen}, ``Distributed convex optimization
  for electric vehicle aggregators,'' \emph{IEEE Transactions on Smart Grid},
  vol.~8, no.~4, pp. 1852--1863, 2017.

\bibitem{facchinei2011decomposition}
F.~Facchinei, V.~Piccialli, and M.~Sciandrone, ``Decomposition algorithms for
  generalized potential games,'' \emph{Computational Optimization and
  Applications}, vol.~50, no.~2, pp. 237--262, 2011.

\bibitem{Millner:2010:Battery_degradation}
A.~{Millner}, ``Modeling lithium ion battery degradation in electric
  vehicles,'' in \emph{2010 IEEE Conference on Innovative Technologies for an
  Efficient and Reliable Electricity Supply}, 2010, pp. 349--356.

\bibitem{fabiani2018mld}
F.~Fabiani and S.~Grammatico, ``A mixed-logical-dynamical model for automated
  driving on highways,'' in \emph{2018 IEEE Conference on Decision and Control
  (CDC)}.\hskip 1em plus 0.5em minus 0.4em\relax IEEE, 2018, pp. 1011--1015.

\bibitem{fabiani2018mvad}
------, ``Multi-vehicle automated driving as a generalized mixed-integer
  potential game,'' \emph{IEEE Transactions on Intelligent Transportation
  Systems}, 2019, (In press).

\bibitem{bemporad1999control}
A.~Bemporad and M.~Morari, ``Control of systems integrating logic, dynamics,
  and constraints,'' \emph{Automatica}, vol.~35, no.~3, pp. 407--427, 1999.

\bibitem{williams2013model}
H.~P. Williams, \emph{Model building in {M}athematical {P}rogramming}.\hskip
  1em plus 0.5em minus 0.4em\relax John Wiley \& Sons, 2013.

\bibitem{facchinei2007generalized}
F.~Facchinei and C.~Kanzow, ``Generalized {N}ash equilibrium problems,''
  \emph{4OR: A Quarterly Journal of Operations Research}, vol.~5, no.~3, pp.
  173--210, 2007.

\bibitem{fabiani2019nash}
F.~Fabiani and A.~Caiti, ``Nash equilibrium seeking in potential games with
  double-integrator agents,'' in \emph{2019 18th European Control Conference
  (ECC)}.\hskip 1em plus 0.5em minus 0.4em\relax IEEE, 2019, pp. 548--553.

\bibitem{sagratella2017algorithms}
S.~Sagratella, ``Algorithms for generalized potential games with mixed-integer
  variables,'' \emph{Computational Optimization and Applications}, pp. 1--29,
  2017.

\bibitem{Melle}
\BIBentryALTinterwordspacing
M.~de~Vent, ``Distributed optimal control of smart electricity grids: Home
  battery and electric vehicle implementation,'' 2018. [Online]. Available:
  \url{http://fse.studenttheses.ub.rug.nl/18698/}
\BIBentrySTDinterwordspacing

\end{thebibliography}

\end{document}